\documentclass[a4paper,12pt]{amsart}
\calclayout

\usepackage{amsmath}
\usepackage{amsfonts}
\usepackage{amssymb}
\usepackage[utf8]{inputenc}
\usepackage{amssymb}
\usepackage{mathrsfs}
\usepackage{amsthm}

\usepackage{enumerate}

\usepackage{graphicx}

\newtheorem{thrm}{Theorem}[section]
\newtheorem{cor}[thrm]{Corollary}
\newtheorem{lemma}[thrm]{Lemma}
\newtheorem{prop}[thrm]{Proposition}

\theoremstyle{definition}

\newtheorem{rem}[thrm]{Remark}

\DeclareMathOperator{\dens}{d}
\DeclareMathOperator{\intt}{int}

\newcommand{\NN}{\mathbb{N}}
\newcommand{\RR}{\mathbb{R}}

\newcommand{\KK}{\mathcal{K}}

\begin{document}

\title[Characterizing function spaces which have the property (B)]{Characterizing function spaces which have the property (B) of Banakh}
\author[M.\ Krupski]{Miko\l aj Krupski}
\address{Universidad de Murcia, Departamento de Matem\'{a}ticas, Campus de Espinardo 30100 Murcia, Spain\\ and \\ Institute of Mathematics\\ University of Warsaw\\ ul. Banacha 2\\
02--097 Warszawa, Poland }
\email{mkrupski@mimuw.edu.pl}

\author[K.\ Kucharski]{Kacper Kucharski}
\address{Institute of Mathematics\\
University of Warsaw\\ Banacha 2\newline 02--097 Warszawa\\
Poland\\}
\email{k.kucharski6@uw.edu.pl}

\author[W.\ Marciszewski]{Witold Marciszewski}
\address{Institute of Mathematics\\
University of Warsaw\\ Banacha 2\newline 02--097 Warszawa\\
Poland\\
ORCID identifier: 0000-0003-3384-5782}
\email{wmarcisz@mimuw.edu.pl}

\thanks{All authors were partially supported by the NCN
(National Science Centre, Poland) research grant no.\ 2020/37/B/ST1/02613.
The first author was also supported by Fundaci\'{o}n S\'{e}neca - ACyT Regi\'{o}n de Murcia project 21955/PI/22, Agencia Estatal de Investigación (Government of Spain)  Project PID2021-122126NB-C32 funded by
MICIU/AEI /10.13039/\\
501100011033/ and FEDER A way of making Europe, and European Union - NextGenerationEU funds through Mar\'{i}a Zambrano fellowship.
}

\date{\today}

\begin{abstract}
A topological space $Y$ has the property (B) of Banakh if there is a countable family $\{A_n:n\in \NN\}$ of closed nowhere dense subsets of $Y$ absorbing all compact subsets of $Y$. In this note we show that the space $C_p(X)$ of continuous real--valued functions on a Tychonoff space $X$ with the topology of pointwise convergence, fails to satisfy the property (B) if and only if the space $X$ has the following property $(\kappa)$: every sequence of disjoint finite subsets of $X$ has a subsequence with point--finite open expansion. Additionally, we provide an analogous characterization for the compact--open topology on $C(X)$. Finally, we give examples of Tychonoff spaces $X$ whose all bounded subsets are finite, yet $X$ fails to have the property $(\kappa)$. This answers a question of Tkachuk.
\end{abstract}

\subjclass[2020]{Primary: 54C35, 54E52, Secondary: 54A10}

\keywords{function space, pointwise convergence topology, compact-open topology, property (B), Banakh property, property $(\kappa)$, $\kappa$-Fr\'echet-Urysohn}

\maketitle

\section{Introduction}

All spaces under consideration are assumed to be Tychonoff.
A space $X$ has \textit{the property (B)} if there is a countable family $\{A_n:n\in \omega\}$ of closed nowhere dense subsets of $X$ such that for any compact subset $K$ of $X$, there is $n\in \omega$ with $K\subseteq A_n$. In particular, such a family $\{A_n:n\in \omega\}$ must cover $X$ and hence the property (B) of $X$ implies that $X$ is meager in itself. This property was introduced in \cite{KM} (see \cite[Definition 5.1]{KM}) as a possible tool for distinguishing the pointwise and the weak topology on the set $C(K)$ of continuous real--valued functions on an infinite compact space $K$. Specifically, it is easy to see that for any infinite--dimensional Banach space $X$, the space $X$ equipped with the weak topology is always (B) (see \cite[Proposition 5.2]{KM}). Given a space $X$, by $C_p(X)$ we denote the space of continuous real--valued functions on $X$ equipped with the pointwise topology. It was established in \cite[Theorem 5.9]{KM} that if $K$ is compact, then $C_p(K)$ enjoys the property (B) if and only if $K$ is not scattered. However, a characterization of spaces $X$ whose function space $C_p(X)$ has the property (B) was left open (see \cite[Problem 5.10]{KM}). In this note we settle this question. In order to formulate our result we first recall some notation.

 Given a family $\mathcal{A}$ of subsets of a space $X$, we say that $\mathcal{A}$ is \textit{point--finite} if for each $x\in X$ the collection $\{A\in \mathcal{A}:x\in A\}$ is finite. The family $\mathcal{A}$ is \textit{strongly point--finite} (\textit{strongly discrete}) there is a point--finite (discrete) family $\{U_A:A\in \mathcal{A}\}$ of open subsets of $X$ satisfying $A\subseteq U_A$ for every $A\in \mathcal{A}$.
Following \cite{S}, we say that a space $X$ has \textit{the property $(\kappa)$} if every pairwise disjoint sequence of finite subsets of $X$ has an infinite strongly point--finite subsequence. The main result of this note reads as follows.

\begin{thrm}\label{main1}
For any space $X$ the following conditions are equivalent:
\begin{enumerate}[(i)]
\item The space $X$ has property $(\kappa)$.
\item The space $C_p(X)$ does not have the property $(B)$.
\end{enumerate}
\end{thrm}

Theorem \ref{main1} should be compared with the following well--known result, proved independently by van Douwen, Pytkeev, and Tkachuk, that characterizes spaces $X$ for which $C_p(X)$ is a Baire space (see \cite[Theorem 6.4.3]{vM} and \cite[Proposition 1.27]{HM}).

\begin{thrm}
 For any space $X$ the following conditions are equivalent:
 \begin{enumerate}[(i)]
  \item Every pairwise disjoint sequence of finite subsets of $X$ has an infinite strongly discrete subsequence.
  \item The space $C_p(X)$ is a Baire space
  \item The space $C_p(X)$ is not meager in itself.
 \end{enumerate}

\end{thrm}

A result analogous to Theorem \ref{main1} is proved for $C_k(X)$, i.e., the space of continuous real--valued functions on $X$ endowed with the compact--open topology (see Section 4 below).

All countable spaces as well as all scattered spaces enjoy the property $(\kappa)$ (see \cite[Corollary 3.8]{S}). In \cite{Tk2} Tkachuk asked if a space $X$ has the property $(\kappa)$ provided all of its bounded subsets are finite (see \cite[Question 4.2]{Tk2}). In Section 5 we provide a negative answer to this question.

\section{Notation}
For a space $X$, we denote by $C(X)$ the set of all real--valued continuous functions on $X$. The constant function equal to $0$ will be denoted by $\mathbf{0}$. By $C_p(X)$ we denote the set $C(X)$ endowed with the pointwise topology, i.e., a basic neighborhood of any $f\in C_p(X)$ is of the form
$$N(f,S,\varepsilon)=\{g\in C_p(X): \forall x\in S\;\;|f(x)-g(x)|<\varepsilon\},$$
where $S$ is a finite subset of $X$ and $\varepsilon>0$.

We write $C_k(X)$ if the compact--open topology on $C(X)$ is considered. The base of the space $C_k(X)$ consists of the sets
$$M(f,K,\varepsilon)=\{g\in C_k(X):\forall x\in K\;\; |f(x)-g(x)|<\varepsilon\},$$
where $f\in C_k(X)$, $\varepsilon>0$ and $K$ is a compact subset of $X$.

A space $X$ is \textit{$\kappa$--Fr\'echet--Urysohn} if for every open subset $U$ of $X$ and every $x\in \overline{U}$, there is a sequence $\{x_n:n\in \omega\}\subseteq U$ that converges to $x$. The letter $\kappa$ here is just a letter and has nothing to do with cardinal numbers. Clearly, every Fr\'echet--Urysohn space is $\kappa$--Fr\'echet--Urysohn. The product $\RR^{\omega_1}$ of uncontably many copies of the real line is $\kappa$--Fr\'echet--Urysohn but not Fr\'echet--Urysohn (see \cite[Theorem 3.1]{Mr}, cf. Theorem \ref{main1} and  the remark following the proof of Theorem 5.9 in \cite{KM}).

\section{The pointwise topology}

\begin{prop}\label{non B implies k}
For any space $X$, if $C_p(X)$ does not have the property $(B)$, then $X$ has the property $(\kappa)$.
\end{prop}
\begin{proof}
Let $\mathcal{S}=\{S_n:n\in \omega\}$ be an arbitrary sequence of pairwise disjoint finite subsets of $X$. We will show that $\mathcal{S}$ has a strongly point--finite subsequence. For any positive integer $k$, define
$$U_k=\left\{f\in C_p(X): (\exists n\in \omega)\;\bigl(f(S_n)\subseteq (k,+\infty)\bigr)\right\}.$$
It is easy to verify that the set $U_k$ is open and dense in $C_p(X)$. Since $C_p(X)$ does not satisfy the property $(B)$, we can find a compact subset $K$ of $C_p(X)$ so that $K\cap U_k\neq \emptyset$ for every $k$. For each $k\geq 1$, take $f_k\in K\cap U_k$ and let $n_k\in\omega$ be a witness for $f_k\in U_k$, i.e., $f_k(S_{n_k})\subseteq (k,+\infty)$. For $k\geq 1$, the set $W_k=f_k^{-1}(k,+\infty)$ is open in $X$ and contains $S_{n_k}$. Let us check that the family $\{W_k:k=1,2,\ldots\}$ is point--finite (and thus $\{S_{n_k}:k=1,2,\ldots\}$ is strongly point--finite). To this end, fix an arbitrary point $x\in X$. Since $K\subseteq C_p(X)$ is compact, the set $\{f(x): f\in K\}$ is bounded in $\RR$. Find an integer $M>0$ so that
\begin{equation}\label{eq1}
\{f(x):f\in K\}\subseteq [-M,M].
\end{equation}
Now, $x\in W_k$ implies $f_k(x)>k$ and since $f_k\in K$, we infer from \eqref{eq1} that $$\{k:x\in W_k\}\subseteq \{1,\ldots, M\}.$$  
\end{proof}

The following useful lemma is known and easy to prove (see, e.g., \cite[6.4.2]{vM}).

\begin{lemma}\label{lemma_modification}
Let $X$ be a space. Suppose that
$H\subseteq X$ is closed and $F\subseteq X$ is finite. Let $\varepsilon>0$ and let $s:F\to \mathbb{R}$ satisfy $|s(x)|< \varepsilon$ for all $x\in H\cap F$. Then exists a continuous function $g:X\to \mathbb{R}$ such that $g(x)=s(x)$ for $x\in F$ and $|g(x)|<\varepsilon$ for $x\in H$.
\end{lemma}

We are ready to prove Theorem \ref{main1}.

\begin{proof}[Proof of Theorem \ref{main1}]
By Proposition \ref{non B implies k} we have $(ii)\Rightarrow (i)$. To prove the converse, suppose that $X$ has the property $(\kappa)$ and let $\{A_n:n=1,2,\ldots\}$ be an increasing sequence of closed nowhere dense subsets of $C_p(X)$. Recursively, construct:
\begin{itemize}
\item an increasing sequence $S_0\subseteq S_1\subseteq\ldots S_n\subseteq\ldots$ of finite subsets of $X$,
\item a sequence $f_1, f_2,\ldots,f_n,\ldots$ of elements of $C_p(X)$ and
\item a sequence $\varepsilon_1,\varepsilon_2,\ldots,\varepsilon_n,\ldots$ of positive reals,
\end{itemize}
so that the following condition is satisfied, for every $n\geq 1$:
\begin{equation}\label{induction}
N(f_{n},S_{n},\varepsilon_{n})\subseteq N\bigl(\mathbf{0},S_{n-1},\tfrac{1}{n}\bigr)\setminus A_n.
\end{equation}

Put $S_0=\emptyset$ and fix $n\geq 1$.
Suppose that $S_{n-1}$ is already defined.
Since $A_n$ is nowhere dense in $C_p(X)$, there is $f_n\in N\bigl(\mathbf{0},S_{n-1},\tfrac{1}{n}\bigr)\setminus A_n$. Using the fact that $A_n$ is closed, we can find a finite set $S_n\supseteq S_{n-1}$ and $\varepsilon_n>0$ so that
$$N(f_{n},S_{n},\varepsilon_{n})\subseteq N\bigl(\mathbf{0},S_{n-1},\tfrac{1}{n}\bigr)\setminus A_n.$$
This finishes the recursive construction.

Let $T_1=S_1$ and $T_{n}=S_{n}\setminus S_{n-1}$ for $n>1$. Apply the property $(\kappa)$ to the sequence $\{T_n:n=1,2,\ldots\}$. It follows, there is an increasing sequence $n_1<\ldots <n_k<\ldots$ and a point--finite family $\{W_k:k=1,2,\ldots\}$ of open sets such that $W_k\supseteq T_{n_k}$, for every $k\geq 1$. We can clearly assume that $n_1>1$.

Given a positive integer $k$, apply Lemma \ref{lemma_modification} with $H=X\setminus W_k$, $F=S_{n_k}$, $\varepsilon=\tfrac{1}{n_k}$ and $s=f_{n_k}\upharpoonright S_{n_k}$. This is possible because $$H\cap F=(X\setminus W_k)\cap S_{n_k}\subseteq S_{n_k-1}$$
and by \eqref{induction},
$$|f_{n_k}(x)|<\varepsilon=\tfrac{1}{n_k} \mbox{, for } x\in S_{n_k-1}.$$
According to Lemma \ref{lemma_modification}, there is a function $g_k\in C_p(X)$ such that
\begin{align}
&g_k(x)=f_{n_k}(x) \mbox{ for } x\in S_{n_k} \quad\mbox{ and } \label{improvement1}\\
&g_k(x)<1/n_k\mbox{ for } x\in X\setminus W_k.\label{improvement2}
\end{align}
We claim that the sequence $(g_k)_k$ converges pointwise to $\mathbf{0}$.  Indeed, pick $x\in X$ and let $\delta>0$. We have $n_1<n_2<\ldots$, so we can find $N_1$ so that $1/n_k<\delta$ for $k>N_1$. Since the family $\{W_k:k=1,2,\ldots\}$ is point--finite, there is $N_2$ with $x\notin W_k$ for $k>N_2$. Put $N=\max\{N_1,N_2\}$ and note that if $k>N$, then
$$g_k(x)<1/n_k<\delta,$$
by \eqref{improvement2} and the choice of $N$. This means that $\lim_{k\to \infty}g_k(x)=0$ as promised.

It follows that the set $K=\{g_k:k=1,2,\ldots\}\cup\{\mathbf{0}\}$ is compact. By \eqref{induction} and \eqref{improvement1}, we have $g_{k}\notin A_{n_k}$ and since $n_1<n_2<\ldots$ and $A_1\subseteq A_2\subseteq\ldots$, we conclude that $K\not\subseteq A_n$ for all $n$.
\end{proof}

\begin{cor}\label{wniosek}
For any space $X$, the following conditions are equivalent:
\begin{enumerate}[(i)]
\item The space $X$ has the property $(\kappa)$.
\item The space $C_p(X)$ is $\kappa$--Fr\'echet--Urysohn.
\item The space $C_p(X)$ is an Ascoli space.
\item The space $C_p(X)$ does not satisfy the property (B).
\end{enumerate}
\end{cor}
\begin{proof}
The equivalence $(i)\Leftrightarrow (ii)$ is a theorem of Sakai from \cite{S}. The implication $(iii)\Rightarrow (ii)$ is due to Gabriyelyan \textit{et. al.} \cite{GGKZ} and $(ii)\Rightarrow (iii)$ by \cite[Theorem 2.5]{G}. Finally, $(i)$ and $(iv)$ are equivalent by our Theorem \ref{main1}.
\end{proof}

For the remark below, we refer the reader to \cite[Theorem 5.9]{KM} or \cite[Corollary 2.7]{G}.

\begin{rem}
If $X$ is a compact space or, more generally, if $X$ is \v{C}ech complete, then all conditions in Corollary \ref{wniosek} are equivalent to the following condition:
\begin{enumerate}[(i)]
\setcounter{enumi}{4}
\item The space $X$ is scattered.
\end{enumerate}
\end{rem}

The following result answers a question asked by Tkachuk (see \cite[Question 4.6]{Tk1}).

\begin{cor}
If $X$ is first--countable pseudocompact non--scattered space, then $C_p(X)$ has the property (B).
\end{cor}
\begin{proof}
 If $C_p(X)$ does not have the property (B), then according to Theorem \ref{main1}, the space $X$ has the property $(\kappa)$. It follows from
 \cite[Theorem 3.19]{Tk2} that $X$ must be then scattered, contradicting our assumption.
\end{proof}

It is shown in \cite[Theorem 5.3]{KM} that a nonempty Fr\'echet--Urysohn space cannot have the property (B). Modifying the reasoning presented in \cite{KM}, we can strengthen this result as follows.

\begin{thrm}\label{theorem kappa-FU implies nonB}
 If a nonempty space $X$ is $\kappa$--Fr\'echet--Urysohn, then $X$ does not have the property (B).
\end{thrm}
\begin{proof}
Striving for a contradiction, suppose that there exists a nonempty $\kappa$--Fr\'{e}chet--Urysohn space $X$ satisfying property (B). Fix closed nowhere dense sets $A_n \subseteq  X$ witnessing it. Without loss of generality, we may assume that $A_n \subseteq  A_{n + 1}$ for all $n$. Observe that the space satisfying (B) cannot have isolated points.

Let $x \in X$ be any point. Since $x\in \overline{X\setminus\{x\}}$ we can apply the $\kappa$--Fr\'echet--Urysohn property of $X$ to find
an injective sequence $(x_n)_{n \in \NN}$ of elements of $X \setminus \{x\}$ converging to $x$. Pick families $\{V_n \colon n \in \NN\}$ and $\{W_n \colon n \in \NN\}$ of open subsets of $X$ satisfying the following conditions for all $n \in \NN$:
\begin{equation}\label{conditions}
 x \in V_n,\;x_n \in W_n \mbox{ and }
V_n \cap W_n = \emptyset.
\end{equation}
For each $n \in \NN$, put $U_n = W_n \setminus A_n.$
Clearly,
$U_n$ is a nonempty open subset of $X$ and $U_n$ is dense in $W_n$. Let
$$U=\bigcup_{n\in \NN}U_n.$$
Since $U_n\subseteq W_n$, we infer from \eqref{conditions} that $x\notin U$.

Let us check that $x\in \overline{U}$. To this end, fix an open neighborhood $P$ of $x$. Since the sequence $(x_n)_{n\in \NN}$ converges to $x$, we can find $n$ so that $x_n\in P$. According to \eqref{conditions} and the fact that $U_n$ is dense in $W_n$, we have $x_n\in P\cap W_n\subseteq P\cap \overline{U_n}$. This means that the latter set is nonempty whence $P\cap U_n\neq \emptyset$, because $P$ is open. In particular $P\cap U\neq\emptyset$ and thus $x\in \overline{U}$. Since $x\in \overline{U}\setminus U$ we can apply the $\kappa$--Fr\'{e}chet--Urysohn property to find an injective sequence $\{y_n: n\in \NN\}\subseteq U$ converging to $x$.

Define
$$K = \{y_n \colon n \in \NN\} \cup \{x\}.$$
The set $K$ is compact and by \eqref{conditions}, for each $n\in \NN$, the intersection $K\cap U_n$ is finite. Hence, $K$ meets infinitely many $U_n$'s. Since the family $\{A_n:n\in \NN\}$ is increasing, it follows that $K\not\subseteq A_n$ for all $n$. This yields a contradiction.
\end{proof}

\begin{rem}
 The above theorem says that the implication $(ii)\Rightarrow (iv)$ in Corollary \ref{wniosek} holds for an arbitrary space.
\end{rem}

\section{The compact--open topology}

In this section we will prove a result analogous to Theorem \ref{main1} for spaces $C_k(X)$ of continuous functions on $X$ equipped with the compact--open topology. We need some notation first. We say that a nonempty family $\KK$ of nonempty subsets of a topological space $X$ is:
\begin{itemize}
\item \emph{Moving off}, if for every compact $L \subseteq  X$, there exists $K \in \KK$ disjoint with $L$;
\item \emph{Compact--finite}, if for every compact $L \subseteq  X$ the set $\{K \in \KK \colon K \cap L \neq \emptyset\}$ is finite;
\item \emph{Strongly compact--finite}, if there exists a compact--finite family $\{U_K \colon K \in \KK\}$ consisting of open subsets of $X$ satisfying $K \subseteq  U_K$ for all $K \in \KK$.
\end{itemize}
In the sequel moving off families will consist only of compact subsets of $X$, however more general approach has been studied in the literature (cf. \cite{Sa2}). The proof of the proposition below is almost the same as the proof of Proposition \ref{non B implies k}.

\begin{prop}\label{proposition for compact-open}
For any space $X$, if $C_k(X)$ does not satisfy the property (B), then every moving off family consisting of nonempty compact subsets of $X$ has a countable infinite strongly compact--finite subfamily.
\end{prop}
\begin{proof}
Fix a moving off family $\KK$ consisting of compact subsets of a space $X$. For every $n \in \NN$, define the set
\[
U_n =\{f \in C_k(X) \colon \exists K \in \KK \;\; f(K) \subseteq  (n, + \infty) \}.
\]
It is clear that each set $U_n$ is open in $C_k(X)$. Since the family $\KK$ is moving off, it easily follows the each $U_n$ is also dense in $C_k(X)$. By the failure of the property (B), we can find a compact subset $A$ of $C_k(X)$ so that $A \cap U_n \neq \emptyset$, for all $n \in \NN$.
For each $n \in \NN$ take $f_n \in A \cap U_n$. Since $f_n\in U_n$, there is $K_n\in \KK$ such that $f_n(K_n)\subseteq (n, + \infty)$. Define $W_n = f_n^{-1}\big( (n, + \infty) \big)$ and notice that $K_n \subseteq  W_n$ for each $n \in \NN$. It suffices to check that $\{W_n \colon n \in \NN\}$ forms a compact--finite family. To this end, fix a compact subset $L$ of $X$. Let us first note the following:
\begin{equation}\label{6}
 \mbox{There is $M>0$ such that $f(L)\subseteq [-M,M]$ for all $f\in A$.}
\end{equation}
Indeed, if $D_n = \{f \in C_k(X) \colon f(L) \subseteq  (-n, n)\}$, then the family $\{D_n:n=1,2\ldots\}$ is an open cover of $C_k(X)$. Since $A\subseteq C_k(X)$ is compact and $D_n\subseteq D_{n+1}$, we can find $M\in \NN$ so that $A\subseteq D_{M}$. This gives \eqref{6}.

Now, if $L\cap W_n\neq\emptyset$ then $f_n(x)>n$ for some $x\in L$. Hence, if $n>M$, then $f_n\notin A$, by \eqref{6}. Since $f_n\in A$ for all $n$, we get
$\{n\in \NN:L\cap W_n\neq\emptyset\}\subseteq \{1,\ldots ,M\}.$
\end{proof}

\begin{thrm}
For any space $X$ the following conditions are equivalent:
\begin{enumerate}[(i)]
\item The space $C_k(X)$ does not satisfy the property $(B)$,
\item Every moving off family consisting of nonempty compact subsets of $X$ has a countable infinite strongly compact--finite subfamily,
\item The space $C_k(X)$ is $\kappa$--Fr\'{e}chet--Urysohn.
\end{enumerate}
\end{thrm}
\begin{proof}
According to Proposition \ref{proposition for compact-open}, we have $(i)\Rightarrow (ii)$. The implication $(ii)\Rightarrow (i)$ is due to Sakai \cite[Theorem 2.3]{Sa2}. Finally, $(iii)\Rightarrow (i)$, by Theorem \ref{theorem kappa-FU implies nonB}.
\end{proof}

\section{Property $(\kappa)$ and the size of bounded subsets}

Recall that a subset $A$ of a topological space $X$ is \textit{bounded} if, for every $f\in C_p(X)$, the set $f(A)$ is bounded in $\mathbb{R}$.

Tkachuk asked the following question (see \cite[Question 4.2]{Tk2}):\smallskip

\emph{Suppose that all bounded subsets of a topological space $X$ are finite. Must $X$ have the property $(\kappa)$?}\smallskip

In this section we will provide a negative answer to this question. For verification of the properties of our counterexample we will use the next two results. The first of them was proved in \cite[Proposition 5.4]{KM} for compact spaces $X$, but exactly the same argument works for Baire spaces (see \cite[Proposition 3.6]{Tk1}).

\begin{prop}\label{B_przelicz}
Let $X$ be a Baire space with a countable family $\mathcal{S}$ of infinite subsets, such that any nonempty open subset of $X$ contains a member of $\mathcal{S}$.
Then $C_p(X)$ has the property (B).
\end{prop}

\begin{prop}\label{density_topologies}
Let $\tau$ be a (Tychonoff) topology on the real line $\mathbb{R}$ with the following properties:
\begin{enumerate}[(a)]
\item $\tau$ is stronger than the Euclidean topology $\tau_e$ on $\mathbb{R}$;
\item every nonempty $U\in\tau$ contains a nonempty $V\in\tau_e$;
\item every set $\{x_n: n\in\omega\}\subset \mathbb{R}$ such that the sequence $(x_n)$ converges in $\tau_e$ is closed in $\tau$.
\end{enumerate}
Then the space $(\mathbb{R},\tau)$ does not have the property $(\kappa)$, but all its bounded subsets are finite.
\end{prop}

\begin{proof}
First, we will check that $(\mathbb{R},\tau)$ is a Baire space. Let $U_n, n\in \omega$, be open dense in $(\mathbb{R},\tau)$. For $n\in \omega$, denote by $V_n$ the interior of $U_n$ in $(\mathbb{R},\tau_e)$. Using properties (a) and (b) one can easily verify that each set $V_n$ is dense in $(\mathbb{R},\tau_e)$. By the Baire category theorem the intersection $\bigcap_{n\in \omega} V_n\subseteq \bigcap_{n\in \omega} U_n$ is dense in $(\mathbb{R},\tau_e)$, and by property (b) is also dense in $(\mathbb{R},\tau)$.

Let $\mathcal{S}$ be a countable base of $(\mathbb{R},\tau_e)$. By property (b) the family $\mathcal{S}$ satisfies the assumption from Proposition \ref{B_przelicz}, hence the space $C_p((\mathbb{R},\tau))$ has the property (B). From Theorem \ref{main1} we conclude that $(\mathbb{R},\tau)$ does not have the property $(\kappa)$.

It remains to verify that each infinite subset $A$ of $\mathbb{R}$ is unbounded in $\tau$.

If $A$ is unbounded in the usual sense in $\mathbb{R}$ then by property (a) the identity function on $\mathbb{R}$ witnesses that $A$ is unbounded in $\tau$.

If $A$ is bounded (in the usual sense), then we can find a strictly monotone sequence $(x_n)$ of points of $A$ converging in $(\mathbb{R},\tau_e)$ to a point $x$. By property (c) the set $F = \{x_n: n\in\omega\}$ is closed in $(\mathbb{R},\tau)$. Take a continuous function $f: (\mathbb{R},\tau) \to [0,1]$ such that $f(x) = 0$ and $f(F) = \{1\}$.

For each $n\in \omega$, pick a point $y_n$ strictly between $x_n$ and $x_{n+1}$ and denote by $I_n$, for $n\ge 1$, an open interval with endpoints $y_{n-1}$ and $y_n$. For $n\ge 1$, put $U_n = I_n\cap f^{-1}((1/2,1])$, and find a continuous function $f_n: (\mathbb{R},\tau) \to [0,n]$ such that $f_n(x_n) = n$ and $f_n(\mathbb{R}\setminus U_n) = \{0\}$.

Finally, define $g: (\mathbb{R},\tau) \to \mathbb{R}$ by
$$g = f + \sum_{n=1}^{\infty} f_n.$$
The function $g$ is well defined since the supports of functions $f_n$ are pairwise disjoint. It is $\tau$--continuous at any point $y\ne x$ because on a sufficiently small neighborhood of $y$ it is a sum of at most 3  $\tau$--continuous functions. The function $g$ is also $\tau$--continuous at $x$ since it coincides with $f$ on a $\tau$--neighborhood $f^{-1}([0,1/2))$ of $x$.

For $n\ge 1$, we have $g(x_n) = n+1$, hence $g$ witnesses that the set $A$ is unbounded.
\end{proof}

To answer Tkachuk's question in the negative, it remains to point out examples of topologies on $\mathbb{R}$ satisfying the assumptions of Proposition \ref{density_topologies}. As such, certain topologies lying strictly between the Euclidean topology and the density topology on $\mathbb{R}$, introduced by O'Malley in \cite{Om1} or \cite{Om2} may serve.\smallskip

For the convenience of the reader let us recall the definitions of some density type topologies on the real line $\mathbb{R}$. 
 Let $A \subset \RR$ be a set measurable with respect to the Lebesgue measure $\lambda$ and let $z \in \RR$. The density of the set $A$ in point $z$ is the number
\[
\dens(A,z) = \lim_{\epsilon \to 0} \frac{\lambda(A \cap [z - \epsilon, z + \epsilon])}{2\epsilon}
\]
if such a limit exists. 

The classical density topology $d$ on $\mathbb{R}$ consists of all measurable subsets of $\mathbb{R}$ which have density $1$ in all its points.

A set $U \in d$ is called {\it almost open} if  its measure is the same as the measure of its Euclidean interior, i.e. $\lambda(U \triangle \intt_{\tau_e}(U)) = 0$. O'Malley proved in \cite{Om1} that the collection of all almost open sets is a Tychonoff topology (cf.  \cite[Corollary 3.5]{Om1}) and called it the {\it almost open} topology $\tau_{a.e.}$. It is clear that the almost open topology is stronger than the Euclidean topology on $\RR$ and satisfies condition (b) from  Proposition \ref{density_topologies}. One can easily verify that it also satisfies condition (c) from this proposition (cf. \cite[Proposition 2.4]{Om1}).

In \cite{Om2} O'Malley defined another topology on $\mathbb{R}$ weaker than the density topology $d$, denoted by $r$, which has a base consisting of all elements of $d$ which are simultaneously $F_\sigma$ and $G_\delta$ in the Euclidean topology.  He proved that this topology is Tychonoff and lies strictly between the topologies $\tau_{a.e.}$ and $d$. Since $r$ is stronger than  $\tau_{a.e.}$ it satisfies condition (c) from  Proposition \ref{density_topologies}. O'Malley also proved that it satisfies condition (b) from this proposition (cf. \cite[Theorem 3.3]{Om2}).

\bibliographystyle{siam}
\bibliography{bib.bib}

\begin{thebibliography}{10}

\bibitem{G}
{\sc S.~Gabriyelyan}, {\em Topological properties of spaces of {B}aire
  functions}, J. Math. Anal. Appl., 478 (2019), pp.~1085--1099.

\bibitem{GGKZ}
{\sc S.~Gabriyelyan, J.~Greb\'{\i}k, J.~K\c{a}kol, and L.~Zdomskyy}, {\em The
  {A}scoli property for function spaces}, Topology Appl., 214 (2016),
  pp.~35--50.

\bibitem{HM}
{\sc R.~C. Haworth and R.~A. McCoy}, {\em Baire spaces}, Dissertationes Math.
  (Rozprawy Mat.), 141 (1977), p.~73.

\bibitem{KM}
{\sc M.~Krupski and W.~Marciszewski}, {\em On the weak and pointwise topologies
  in function spaces ii}, J. Math. Anal. Appl., 452 (2017), pp.~646--658.

\bibitem{Mr}
{\sc S.~Mr\'{o}wka}, {\em Mazur theorem and {$m$}-adic spaces}, Bull. Acad.
  Polon. Sci. S\'{e}r. Sci. Math. Astronom. Phys., 18 (1970), pp.~299--305.

\bibitem{Om2}
{\sc R.~J. O'Malley}, {\em Approximately differentiable functions: the {$r$}
  topology}, Pacific J. Math., 72 (1977), pp.~207--222.

\bibitem{Om1}
\leavevmode\vrule height 2pt depth -1.6pt width 23pt, {\em Approximately
  continuous functions which are continuous almost everywhere}, Acta Math.
  Acad. Sci. Hungar., 33 (1979), pp.~395--402.

\bibitem{S}
{\sc M.~Sakai}, {\em Two properties of {$C_p(X)$} weaker than the {F}r\'{e}chet
  {U}rysohn property}, Topology Appl., 153 (2006), pp.~2795--2804.

\bibitem{Sa2}
\leavevmode\vrule height 2pt depth -1.6pt width 23pt, {\em
  {$\kappa$}-{F}r\'{e}chet {U}rysohn property of {$C_k(X)$}}, Topology Appl.,
  154 (2007), pp.~1516--1520.

\bibitem{Tk1}
{\sc V.~V. Tkachuk}, {\em Many {E}berlein-{G}rothendieck spaces have no
  non-trivial convergent sequences}, Eur. J. Math., 4 (2018), pp.~664--675.

\bibitem{Tk2}
\leavevmode\vrule height 2pt depth -1.6pt width 23pt, {\em A note on
  {$\kappa$}-{F}r\'{e}chet-{U}rysohn property in function spaces}, Bull. Belg.
  Math. Soc. Simon Stevin, 28 (2021), pp.~123--132.

\bibitem{vM}
{\sc J.~van Mill}, {\em The infinite-dimensional topology of function spaces},
  vol.~64 of North-Holland Mathematical Library, North-Holland Publishing Co.,
  Amsterdam, 2001.

\end{thebibliography}
\end{document}